\documentclass[11pt]{amsart}

\usepackage{amsmath,amssymb,amsfonts,mathtools}
\usepackage{enumitem}
\usepackage{hyperref}
\usepackage{microtype}
\usepackage{mathrsfs}

\numberwithin{equation}{section}

\newtheorem{theorem}{Theorem}[section]
\newtheorem{lemma}[theorem]{Lemma}

\theoremstyle{definition}
\newtheorem{definition}[theorem]{Definition}

\newcommand{\R}{\mathbb{R}}
\newcommand{\N}{\mathbb{N}}
\newcommand{\W}{W_0^{1,p}(\Omega)}
\newcommand{\J}{\mathcal{J}}

\title{A $p$-Laplacian Schr\"odinger--Maxwell system with rapidly growing nonlinearities}

\author[G. da Silva]{Genival da Silva}
\address{Department of Mathematical, Physical, and Engineering Sciences,
Texas A\&M University--San Antonio, San Antonio, TX, USA}
\email{gdasilva@tamusa.edu}



\begin{document}

\begin{abstract}
Let $\Omega\subset\R^N$, $N\ge2$, be a bounded domain and let $1<p<\infty$.
Inspired by the exponential case treated in \cite{BO2026}, we study the
quasilinear Schr\"odinger--Maxwell system
\[
\begin{cases}
-\Delta_p u+\psi G'(u)=f &\text{in }\Omega,\\
-\Delta_p\psi=G(u) &\text{in }\Omega,\\
u=\psi=0 &\text{on }\partial\Omega,
\end{cases}
\]
where $G\in C^1(\R)$ is even, convex, nonnegative and nontrivial, with
$G(0)=0$, and satisfies
\[
G(t)\lesssim 1+|G'(t)|
\qquad\text{for large }|t|.
\]
Under the assumption
\[
|f|\ln(1+|f|)\in L^1(\Omega),
\]
we prove the existence of a finite-energy weak solution and show that the
solution is a saddle point of the associated functional.
\end{abstract}

\maketitle

\section{Introduction}

Schr\"odinger--Maxwell systems arise naturally in the study of the interaction
between a quantum mechanical field and its associated electromagnetic field.
A variational treatment of such systems goes back to the work of Benci and
Fortunato \cite{BenciFortunato}, where the corresponding energy functional,
although unbounded both from above and from below, was shown to possess
critical points of saddle type.

In the framework of Dirichlet problems with low summability data, Boccardo and Orsina
\cite{BO2018} discovered an important regularizing effect produced by the
coupling between the two equations: finite-energy solutions may exist even
when the source does not belong to the natural dual space. This phenomenon has
subsequently been investigated for several classes of Schr\"odinger--Maxwell
systems, including semilinear systems and their saddle-point structure
\cite{BO2020}, $p$-Laplacian systems \cite{Durastanti2019},
Kirchhoff--Schr\"odinger--Maxwell problems \cite{BOKirchhoff}, and systems
involving singular nonlinearities \cite{BOSingular}.

More recently, Boccardo and Orsina \cite{BO2026} considered a
Schr\"odinger--Maxwell system with exponential nonlinearities and proved the
existence of finite-energy saddle-point solutions under the assumption
$|f|\ln(1+|f|)\in L^1(\Omega)$. The purpose of the present paper is to combine
the quasilinear diffusion appearing in \cite{Durastanti2019} with a general
class of rapidly growing convex couplings suggested by the exponential model
of \cite{BO2026}.

We consider
\begin{equation}\label{main-system}\tag{*}
\begin{cases}
-\Delta_p u+\psi G'(u)=f &\text{in }\Omega,\\
-\Delta_p\psi=G(u) &\text{in }\Omega,\\
u=\psi=0 &\text{on }\partial\Omega,
\end{cases}
\end{equation}
where
\[
\Delta_pv=\operatorname{div}(|Dv|^{p-2}Dv),
\qquad 1<p<\infty.
\]
Throughout the paper we assume
\begin{equation}\label{G-assumptions}
G\in C^1(\R),\qquad G \text{ is even, convex, nonnegative and }G\not\equiv 0,
\qquad G(0)=0,
\end{equation}
and that there exist $C_0>0$ and $t_0\ge0$ such that
\begin{equation}\label{G-growth}
G(t)\le C_0\bigl(1+|G'(t)|\bigr),
\qquad |t|\ge t_0.
\end{equation}
Since $G$ is even and convex, $G'$ is odd and
\begin{equation}\label{sign-Gprime}
G'(t)t\ge0\qquad\text{for every }t\in\R.
\end{equation}
Typical examples of nonlinearities satisfying our assumptions are functions of the form
\[
G(t)=e^{a|t|}-a|t|-1,
\qquad a>0.
\]
Another example is
\[
G(t)=\cosh(at)-1,
\qquad a>0.
\]
Finally, a super-exponential example is
\[
G(t)=e^{t^2}-1.
\]
Notice that the structural condition \eqref{G-growth} links the two equations of \ref{main-system}: at large
values of $|u|$, the source $G(u)$ in the second equation is controlled by the
absorption $G'(u)$ appearing in the first. The exponential function
\[
G(t)=e^{|t|}-|t|-1
\]
is the motivating example, but no explicit formula for $G$ is used below.

The source $f$ satisfies
\begin{equation}\label{datum-assumption}
f \in L^1(\Omega),
\qquad
|f|\ln(1+|f|) \in L^1(\Omega),
\end{equation}
The natural functional related to System \ref{main-system} is
\begin{equation}\label{functional}
\J(v,\varphi)
=
\frac1p\int_\Omega |Dv|^p
-\frac1p\int_\Omega |D\varphi|^p
+\int_\Omega \varphi^+G(v)
-\int_\Omega fv.
\end{equation}
It is convex in $v$ and concave in $\varphi$ after reduction to the positive
cone in the second variable. Since $G$ may grow very rapidly,
\eqref{functional} need not be finite or $C^1$ on the whole product space.
We therefore exploit its convex--concave structure directly.

We record two consequences of \eqref{G-assumptions}--\eqref{G-growth}
which will be used later in the text.

\begin{lemma}\label{global-growth-G}
There exists $C>0$ such that
\begin{equation}\label{global-G-Gprime}
0\le G(t)\le C\bigl(1+|G'(t)|\bigr)
\qquad\text{for every }t\in\R.
\end{equation}
Moreover, there exist $a>0$, $c>0$ and $T>0$ such that
\begin{equation}\label{exp-lower-G}
e^{a|t|}\le c\bigl(1+G(t)\bigr)
\qquad\text{for }|t|\ge T.
\end{equation}
Consequently, after changing $c$,
\begin{equation}\label{exp-lower-G-global}
e^{a|t|}\le c\bigl(1+G(t)\bigr)
\qquad\text{for every }t\in\R.
\end{equation}
\end{lemma}

\begin{proof}
The first assertion follows from \eqref{G-growth} after enlarging the constant
on the compact interval $[-t_0,t_0]$.

Since $G$ is nontrivial, even and convex with $G(0)=0$, it is unbounded on
$[0,\infty)$. For $t\ge t_0$, convexity gives $G'(t)\ge0$, and
\eqref{G-growth} yields
\[
G'(t)\ge \frac1{C_0}G(t)-1.
\]
Choose $T\ge t_0$ so large that $G(T)>2C_0$. Then, for $t\ge T$,
\[
G'(t)\ge \frac1{2C_0}G(t).
\]
Integrating the differential inequality gives
\[
G(t)\ge G(T)e^{(t-T)/(2C_0)},\qquad t\ge T.
\]
Evenness gives the same estimate for negative $t$. This proves
\eqref{exp-lower-G}; the global version follows by enlarging the constant on
$[-T,T]$.
\end{proof}

\begin{lemma}\label{scaled-young}
For every $a>0$ there exists $C_a>0$ such that
\begin{equation}\label{scaled-young-eq}
st\le C_a\,s\ln(1+s)+e^{at}+C_a,
\qquad s,t\ge0.
\end{equation}
Moreover, for every $\theta>0$ there exists $C_\theta>0$ such that
\begin{equation}\label{scaled-log-young}
s\ln(1+t)
\le C_\theta s\ln(1+s)+(1+t)^\theta+C_\theta,
\qquad s,t\ge0.
\end{equation}
\end{lemma}

\begin{proof}
Both estimates follow from
\[
xy\le x\ln(1+x)+e^y-1,
\qquad x,y\ge0,
\]
after a scaling of $x$ and $y$.
\end{proof}

\begin{definition}
A pair $(u,\psi)\in\W\times\W$ is a weak solution of
\eqref{main-system} if $\psi\ge0$ a.e.,
\[
\psi|G'(u)|\in L^1(\Omega),
\qquad
G(u)\in L^1(\Omega),
\]
and
\begin{equation}\label{weak-first}
\int_\Omega |Du|^{p-2}Du\cdot Dv
+
\int_\Omega \psi G'(u)v
=
\int_\Omega fv
\end{equation}
for every $v\in\W\cap L^\infty(\Omega)$, while
\begin{equation}\label{weak-second}
\int_\Omega |D\psi|^{p-2}D\psi\cdot D\varphi
=
\int_\Omega G(u)\varphi
\end{equation}
for every $\varphi\in\W$.
\end{definition}

For a fixed nonnegative $\psi\in\W$, define
\[
\mathcal D_\psi
=
\left\{
v\in\W:
fv\in L^1(\Omega),\;\psi G(v)\in L^1(\Omega)
\right\}.
\]
Whenever $G(u)\in L^1$, define
\[
\mathcal E_u
=
\left\{
\varphi\in\W:
\varphi^+G(u)\in L^1(\Omega)
\right\}.
\]

We are now in a position to state our main result.

\begin{theorem}\label{main-thm}
Assume \eqref{G-assumptions}, \eqref{G-growth}, and
\eqref{datum-assumption}. Then there exists a weak solution
\[
(u,\psi)\in\W\times\W
\]
of \eqref{main-system}. Moreover, $(u,\psi)$ satisfies
\begin{equation}\label{saddle}
\J(u,\varphi)
\le
\J(u,\psi)
\le
\J(v,\psi)
\end{equation}
for every $\varphi\in\mathcal E_u$ and every $v\in\mathcal D_\psi$.
\end{theorem}

\medskip

{\bf Notation.} 
\medskip

The following truncation function will be used throughout the text.
\[
T_k (s)= \max ( -k, \min ( s, k)),\quad
\]
Throughout the paper, the letter $C$ denotes a positive constant that may vary from line to line. We denote the Holder conjugate of $p$ by $p'$.

\section{Proof of Theorem~\ref{main-thm}}

Fix $n\in\N$ and let
$
f_n=T_n(f).
$
We first construct solutions $(u_{n},\psi_{n})$ of the approximated system
\begin{equation}\label{ap-sys}\tag{**}
\begin{cases}
-\Delta_pu_n+\psi_nG'(u_n)=f_n &\text{in }\Omega,\\
-\Delta_p\psi_n=G(u_n) &\text{in }\Omega,\\
u_n=\psi_n=0 &\text{on }\partial\Omega.
\end{cases}
\end{equation}

Fix $\varphi\in L^{p'}(\Omega)$. We show that the problem
\begin{equation}\label{fixed-phi-eq}
-\Delta_p w+\varphi^+G'(w)=f_n,
\qquad w\in\W,
\end{equation}
has a unique weak solution which is bounded independently of $\varphi$.

For $m>0$, set
\[
g_m(t)=T_m(G'(t)).
\]
Since $G$ is convex, $G'$ is nondecreasing, and therefore $g_m$ is
continuous, bounded, and nondecreasing. Moreover, by
\eqref{sign-Gprime},
\[
g_m(t)t\ge0
\qquad\text{for every }t\in\mathbb R.
\]
Consider the truncated problem
\begin{equation}\label{truncated-fixed-phi}
-\Delta_p w_m+\varphi^+g_m(w_m)=f_n,
\qquad w_m\in\W.
\end{equation}
Define
\[
\mathcal A_m:\W\longrightarrow\W^*
\]
by
\[
\langle\mathcal A_m(v),\eta\rangle
=
\int_\Omega |Dv|^{p-2}Dv\cdot D\eta
+
\int_\Omega\varphi^+g_m(v)\eta.
\]
This operator is well defined. Indeed, since $g_m$ is bounded and
$\varphi\in L^{p'}(\Omega)$,
\[
\left|
\int_\Omega\varphi^+g_m(v)\eta
\right|
\le
m\|\varphi\|_{p'}\|\eta\|_p
\le
C m\|\varphi\|_{p'}\|\eta\|_{\W}.
\]
The operator $\mathcal A_m$ is continuous and monotone. In fact,
for $v_1,v_2\in\W$,
\begin{align*}
&\langle
\mathcal A_m(v_1)-\mathcal A_m(v_2),v_1-v_2
\rangle
\\
&=
\int_\Omega
\big(
|Dv_1|^{p-2}Dv_1-|Dv_2|^{p-2}Dv_2
\big)
\cdot(Dv_1-Dv_2)
\\
&\quad+
\int_\Omega
\varphi^+
\big(g_m(v_1)-g_m(v_2)\big)(v_1-v_2)
\ge0.
\end{align*}
Moreover,
\[
\langle\mathcal A_m(v),v\rangle
=
\int_\Omega|Dv|^p
+
\int_\Omega\varphi^+g_m(v)v
\ge
\|Dv\|_p^p,
\]
so $\mathcal A_m$ is coercive. Since the $p$-Laplacian is strictly
monotone, $\mathcal A_m$ is strictly monotone. Hence, by the
Browder--Minty theorem, \eqref{truncated-fixed-phi} has a unique weak
solution $w_m\in\W$.

We now obtain an $L^\infty$ estimate which is independent of both
$m$ and $\varphi$. Let $z\in\W$ be the solution of
\begin{equation}\label{comparison-z}
-\Delta_p z=|f_n|,
\qquad z=0\quad\text{on }\partial\Omega.
\end{equation}
Since $f_n\in L^\infty(\Omega)$, standard boundedness estimates for
the $p$-Laplacian give
\[
z\in\W\cap L^\infty(\Omega),
\qquad z\ge0.
\]
We claim that
\[
-z\le w_m\le z
\qquad\text{a.e. in }\Omega.
\]

Indeed, subtracting \eqref{comparison-z} from
\eqref{truncated-fixed-phi} and testing with
$(w_m-z)^+$ gives
\begin{align*}
&
\int_{\{w_m>z\}}
\big(
|Dw_m|^{p-2}Dw_m-|Dz|^{p-2}Dz
\big)
\cdot(Dw_m-Dz)
\\
&\qquad+
\int_\Omega
\varphi^+g_m(w_m)(w_m-z)^+
\\
&=
\int_\Omega(f_n-|f_n|)(w_m-z)^+
\le0.
\end{align*}
On $\{w_m>z\}$ we have $w_m>0$, because $z\ge0$, and hence
$g_m(w_m)\ge0$. Both terms on the left-hand side are therefore
nonnegative. It follows that
\[
(w_m-z)^+=0.
\]
Thus $w_m\le z$. Testing in the same way with
$(-z-w_m)^+$ yields $w_m\ge-z$. Consequently,
\begin{equation}\label{uniform-w-Linfty}
\|w_m\|_{L^\infty(\Omega)}
\le
\|z\|_{L^\infty(\Omega)}
=:M_n,
\end{equation}
where $M_n$ is independent of $m$ and $\varphi$.

Testing \eqref{truncated-fixed-phi} with $w_m$ and using
$g_m(w_m)w_m\ge0$, we also obtain
\[
\int_\Omega|Dw_m|^p
\le
\int_\Omega |f_n||w_m|
\le
\|f_n\|_\infty |\Omega|M_n.
\]
Hence
\begin{equation}\label{uniform-w-bound}
\|w_m\|_{\W}
+
\|w_m\|_{L^\infty(\Omega)}
\le C_n,
\end{equation}
with $C_n$ independent of $m$ and $\varphi$.

Finally, since $G'$ is continuous, it is bounded on the compact
interval $[-M_n,M_n]$. Choose
\[
m>
\max_{|t|\le M_n}|G'(t)|.
\]
By \eqref{uniform-w-Linfty},
\[
g_m(w_m)=G'(w_m)
\qquad\text{a.e. in }\Omega.
\]
Thus $w_m$ is in fact a weak solution of the original problem
\eqref{fixed-phi-eq}. We denote it by $w$.

The solution $w$ is unique. Indeed, if $w_1,w_2$ are two solutions, then,
since both are bounded, all the terms below are integrable, and
testing the difference of their equations with $w_1-w_2$ gives
\begin{align*}
&
\int_\Omega
\big(
|Dw_1|^{p-2}Dw_1-|Dw_2|^{p-2}Dw_2
\big)
\cdot(Dw_1-Dw_2)
\\
&\quad+
\int_\Omega
\varphi^+
\big(G'(w_1)-G'(w_2)\big)(w_1-w_2)
=0.
\end{align*}
Both terms are nonnegative, and the first is strictly monotone.
Therefore $w_1=w_2$.

For this $w$, solve
\begin{equation}\label{fixed-w-eq}
-\Delta_p\Phi=G(w),
\qquad \Phi\in\W.
\end{equation}
Since $w\in L^\infty(\Omega)$ and $G$ is continuous,
$G(w)\in L^\infty(\Omega)$. Hence \eqref{fixed-w-eq} has a unique solution
\[
\Phi\in\W\cap L^\infty(\Omega),
\qquad \Phi\ge0.
\]
Moreover,
\begin{equation}\label{Phi-bound}
\|\Phi\|_{\W}+\|\Phi\|_\infty\le C_n,
\end{equation}
where $C_n$ is independent of the original $\varphi$. Define
\[
S_n:L^{p'}(\Omega)\longrightarrow L^{p'}(\Omega),
\qquad S_n(\varphi)=\Phi.
\]

\begin{lemma}\label{Schauder-map}
For fixed $n$, the map
$S_n:L^{p'}(\Omega)\to L^{p'}(\Omega)$ is continuous and compact and maps a
sufficiently large closed ball into itself.
\end{lemma}

\begin{proof}
The invariant-ball property follows immediately from \eqref{Phi-bound} and
$|\Omega|<\infty$.

Let $\varphi_j\to\varphi$ in $L^{p'}(\Omega)$, and let $w_j$ and $w$ be the
solutions of \eqref{fixed-phi-eq} corresponding to $\varphi_j$ and $\varphi$,
respectively. By \eqref{uniform-w-bound}, both $w_j$ and $w$ are bounded in
$\W\cap L^\infty(\Omega)$ by a constant depending only on $n$.

Subtracting the two equations and testing with $w_j-w$ gives
\begin{align*}
&\int_\Omega
\bigl(|Dw_j|^{p-2}Dw_j-|Dw|^{p-2}Dw\bigr)
\cdot(Dw_j-Dw)\\
&\quad+
\int_\Omega
\varphi_j^+\bigl(G'(w_j)-G'(w)\bigr)(w_j-w)\\
&=
-\int_\Omega
(\varphi_j^+-\varphi^+)G'(w)(w_j-w).
\end{align*}
Since $G'$ is nondecreasing, the second term on the left is nonnegative.
Moreover, $G'(w)$ is bounded because $w$ is bounded. Hence, by H\"older's and
Poincar\'e's inequalities,
\[
\left|
\int_\Omega
(\varphi_j^+-\varphi^+)G'(w)(w_j-w)
\right|
\le
C_n\|\varphi_j^+-\varphi^+\|_{p'}\|w_j-w\|_p
\longrightarrow0,
\]
because $(w_j)$ is bounded in $\W$ and
$\varphi_j^+\to\varphi^+$ in $L^{p'}(\Omega)$. Therefore
\[
\int_\Omega
\bigl(|Dw_j|^{p-2}Dw_j-|Dw|^{p-2}Dw\bigr)
\cdot(Dw_j-Dw)
\longrightarrow0.
\]
By the (strict) monotonicity of the $p$-Laplacian,
\[
w_j\to w
\qquad\text{strongly in }\W.
\]
After passing to a subsequence, $w_j\to w$ a.e.; since the sequence is
uniformly bounded in $L^\infty(\Omega)$, dominated convergence yields
\[
G(w_j)\to G(w)
\qquad\text{strongly in }L^{p'}(\Omega).
\]

If $\Phi_j=S_n(\varphi_j)$ and $\Phi=S_n(\varphi)$, subtracting the equations
for $\Phi_j$ and $\Phi$ and using again the strict monotonicity of the
$p$-Laplacian gives
\[
\Phi_j\to\Phi
\qquad\text{strongly in }\W,
\]
and hence in $L^{p'}(\Omega)$ because of the uniform $L^\infty$ bound. This
proves continuity.

Finally, let $(\varphi_j)$ be bounded in $L^{p'}(\Omega)$. By
\eqref{Phi-bound}, the sequence $(S_n(\varphi_j))$ is bounded in
$\W\cap L^\infty(\Omega)$. Rellich's theorem gives, after extraction, strong
convergence in $L^1(\Omega)$. Interpolating this convergence with the uniform
$L^\infty$ bound yields strong convergence in $L^{p'}(\Omega)$. Thus $S_n$ is
compact.
\end{proof}

By Schauder's fixed-point theorem there exists
$\psi_n\in L^{p'}(\Omega)$ such that $S_n(\psi_n)=\psi_n$. Let $u_n$ be the
solution of \eqref{fixed-phi-eq} corresponding to $\varphi=\psi_n$. Since the
range of $S_n$ is contained in $\W\cap L^\infty(\Omega)$, we have
$\psi_n\in\W\cap L^\infty(\Omega)$. Therefore $(u_n,\psi_n)$ solves
\eqref{ap-sys} and $\psi_n\ge0$.


Set
\[
H(t)=\operatorname{sgn}(t)\ln(1+|t|).
\]
Since $G'(t)H(t)\ge0$ by \eqref{sign-Gprime}, testing the first equation in
\eqref{ap-sys} with $H(u_n)$ gives
\begin{equation}\label{weighted-basic}
\int_\Omega\frac{|Du_n|^p}{1+|u_n|}
\le
\int_\Omega|f|\ln(1+|u_n|).
\end{equation}
Let
\[
z_n=(1+|u_n|)^{1/p'}-1.
\]
Then
\begin{equation}\label{z-gradient}
|Dz_n|^p
=
\left(\frac1{p'}\right)^p
\frac{|Du_n|^p}{1+|u_n|}.
\end{equation}
Choose $0<\theta<p-1$. By \eqref{weighted-basic} and \eqref{scaled-log-young},
\[
\|Dz_n\|_p^p
\le C+C\int_\Omega(1+|u_n|)^\theta
=C+C\int_\Omega(1+z_n)^{\theta p'}.
\]
Since $\theta p'<p$, Poincar\'e and Young inequalities give
\begin{equation}\label{z-uniform}
\|z_n\|_{\W}\le C,
\end{equation}
and therefore
\begin{equation}\label{weighted-uniform}
\int_\Omega\frac{|Du_n|^p}{1+|u_n|}\le C.
\end{equation}
After extraction, $z_n\to z$ strongly in $L^p$ and a.e. Hence
\begin{equation}\label{u-ae}
u_n\to u\qquad\text{a.e. in }\Omega.
\end{equation}
Furthermore,
\begin{equation}\label{level-measure}
|\{|u_n|>k\}|
\le
\frac{C}{\big((1+k)^{1/p'}-1\big)^p}
\longrightarrow0
\end{equation}
uniformly in $n$.

For $k\ge0$ and $\varepsilon>0$, set
\[
\eta_{k,\varepsilon}(t)
=
\min\left\{\frac{(|t|-k)^+}{\varepsilon},1\right\}
\operatorname{sgn}(t).
\]
Testing the first equation in \eqref{ap-sys} with $\eta_{k,\varepsilon}(u_n)$ and using
$G'(t)\operatorname{sgn}(t)=|G'(t)|$, we obtain after
$\varepsilon\downarrow0$
\begin{equation}\label{tail-coupled}
0\le
\int_{\{|u_n|>k\}}\psi_n|G'(u_n)|
\le
\int_{\{|u_n|>k\}}|f|.
\end{equation}
In particular,
\begin{equation}\label{coupled-L1}
\int_\Omega\psi_n|G'(u_n)|\le\|f\|_1.
\end{equation}
Testing the second equation  in \eqref{ap-sys}  with $\psi_n$ and using
\eqref{global-G-Gprime},
\begin{align*}
\|D\psi_n\|_p^p
&=\int_\Omega\psi_nG(u_n)\\
&\le C\int_\Omega\psi_n
+C\int_\Omega\psi_n|G'(u_n)|\\
&\le C\|D\psi_n\|_p+C.
\end{align*}
Thus
\begin{equation}\label{psi-Wp}
\|\psi_n\|_{\W}\le C.
\end{equation}
Consequently, after extraction,
\begin{equation}\label{psi-conv}
\psi_n\rightharpoonup\psi\quad\text{in }\W,
\qquad
\psi_n\to\psi\quad\text{strongly in }L^1(\Omega),
\qquad
\psi_n\to\psi\quad\text{a.e.},
\end{equation}
with $\psi\ge0$.

We now show that 
\begin{equation}\label{coupled-strong-target}
\psi_nG'(u_n)\to\psi G'(u)
\quad\text{strongly in }L^1(\Omega).
\end{equation}
For a measurable $E\subset\Omega$, \eqref{tail-coupled} gives
\begin{align*}
\int_E\psi_n|G'(u_n)|
&\le
\int_{\{|u_n|>k\}}|f|
+
\sup_{|t|\le k}|G'(t)|\int_E\psi_n.
\end{align*}
By \eqref{level-measure} and the absolute continuity of the integral of $f$,
the first term is uniformly small for large $k$. The second is uniformly
small for small $|E|$ because $\psi_n\to\psi$ strongly in $L^1$. Thus
$\{\psi_nG'(u_n)\}$ is equiintegrable. Together with the a.e. convergence,
Vitali's theorem proves \eqref{coupled-strong-target}.

We next prove
\begin{equation}\label{G-strong-target}
G(u_n)\to G(u)
\quad\text{strongly in }L^1(\Omega).
\end{equation}
Testing the first equation in \eqref{ap-sys}  with $T_k(u_n)$ and using
$G'(u_n)T_k(u_n)\ge0$ gives
\begin{equation}\label{Tk-gradient}
\int_\Omega|DT_k(u_n)|^p\le k\|f\|_1.
\end{equation}
Testing the second equation in \eqref{ap-sys} with $|T_k(u_n)|$ yields
\begin{align*}
\int_\Omega|T_k(u_n)|G(u_n)
&=\int_\Omega |D\psi_n|^{p-2}D\psi_n\cdot D|T_k(u_n)|\\
&\le \|D\psi_n\|_p^{p-1}\|DT_k(u_n)\|_p
\le Ck^{1/p}.
\end{align*}
Therefore
\begin{equation}\label{G-tail}
\int_{\{|u_n|>k\}}G(u_n)
\le \frac{C}{k^{1/p'}}.
\end{equation}
For measurable $E$,
\[
\int_EG(u_n)
\le \frac{C}{k^{1/p'}}+
\sup_{|t|\le k}G(t)\,|E|.
\]
Hence $\{G(u_n)\}$ is equiintegrable, and Vitali gives
\eqref{G-strong-target}. By Lemma~\ref{global-growth-G},
\eqref{exp-lower-G-global}, the family $\{e^{a|u_n|}\}$ is also
equiintegrable and
\begin{equation}\label{exp-a-strong}
e^{a|u_n|}\to e^{a|u|}
\quad\text{strongly in }L^1(\Omega).
\end{equation}

Testing the first equation in \eqref{ap-sys}  with $u_n$ and using
$G'(u_n)u_n\ge0$, we find
\[
\int_\Omega|Du_n|^p\le\int_\Omega|f||u_n|.
\]
Using \eqref{scaled-young-eq} with the constant $a$ from
Lemma~\ref{global-growth-G},
\begin{align*}
\int_\Omega|Du_n|^p
&\le C\int_\Omega|f|\ln(1+|f|)
+\int_\Omega e^{a|u_n|}+C\\
&\le C.
\end{align*}
Thus
\begin{equation}\label{u-Wp}
\|u_n\|_{\W}\le C,
\end{equation}
and, up to a subsequence,
\begin{equation}\label{u-weak}
u_n\rightharpoonup u\quad\text{in }\W.
\end{equation}

We claim
\begin{equation}\label{fnun-strong}
f_nu_n\to fu\quad\text{strongly in }L^1(\Omega).
\end{equation}
Indeed, $f_nu_n\to fu$ a.e. and, by \eqref{scaled-young-eq},
\[
|f_nu_n|
\le C|f|\ln(1+|f|)+e^{a|u_n|}+C.
\]
The right-hand side is equiintegrable by \eqref{datum-assumption} and
\eqref{exp-a-strong}. Vitali proves \eqref{fnun-strong}.

We use the following standard lemma; see, for
example, \cite{BoccardoMurat}.

\begin{lemma}\label{flux-lemma}
Let $v_n,v\in\W$ satisfy
\[
v_n\rightharpoonup v\quad\text{in }\W,
\qquad
v_n\to v\quad\text{a.e.},
\]
and assume
\[
-\operatorname{div}|Dv_n|^{p-2}Dv_n=F_n
\]
in the sense of bounded $\W$ test functions, with
$F_n\to F$ strongly in $L^1(\Omega)$. Then, after extraction,
\[
Dv_n\to Dv\quad\text{a.e. in }\Omega,
\qquad
|Dv_n|^{p-2}Dv_n\rightharpoonup |Dv|^{p-2}Dv
\quad\text{in }L^{p'}(\Omega;\R^N).
\]
Consequently $-\operatorname{div}|Dv|^{p-2}Dv=F$ in the sense of bounded $\W$ test
functions.
\end{lemma}

Recall that
\[
\psi_nG'(u_n)\to\psi G'(u)
\quad\text{strongly in }L^1(\Omega),
\]
and $f_n\to f$ strongly in $L^1$. Therefore
\[
F_n:=f_n-\psi_nG'(u_n)
\to
F:=f-\psi G'(u)
\quad\text{strongly in }L^1.
\]
Since $-\Delta_pu_n=F_n$, Lemma~\ref{flux-lemma} yields
\[
Du_n\to Du\quad\text{a.e.},
\qquad
|Du_n|^{p-2}Du_n\rightharpoonup |Du|^{p-2}Du
\quad\text{in }L^{p'}.
\]
Passing to the limit against $v\in\W\cap L^\infty(\Omega)$ gives
\[
\int_\Omega |Du|^{p-2}Du\cdot Dv
+\int_\Omega\psi G'(u)v
=\int_\Omega fv.
\]

Now, recall that
\[
G(u_n)\to G(u)\quad\text{strongly in }L^1(\Omega).
\]
Applying Lemma~\ref{flux-lemma} to
$-\Delta_p\psi_n=G(u_n)$ gives
\[
D\psi_n\to D\psi\quad\text{a.e.},
\qquad
|D\psi_n|^{p-2}D\psi_n\rightharpoonup |D\psi|^{p-2}D\psi
\quad\text{in }L^{p'}.
\]
Thus
\begin{equation}\label{psi-bounded-tests}
\int_\Omega |D\psi|^{p-2}D\psi\cdot D\varphi
=\int_\Omega G(u)\varphi
\end{equation}
for every $\varphi\in\W\cap L^\infty(\Omega)$. 

If $\varphi\ge0$ belongs to
$\W$ only, use $T_k(\varphi)$ in \eqref{psi-bounded-tests}. Since
$DT_k(\varphi)\to D\varphi$ strongly in $L^p$ and
$T_k(\varphi)\uparrow\varphi$, one may pass to the limit on both sides. A
general $\varphi$ follows by writing $\varphi=\varphi^+-\varphi^-$. Hence
\eqref{weak-second} holds for every $\varphi\in\W$.

We complete the proof of Theorem~\ref{main-thm} by proving the next two lemmas.
\begin{lemma}\label{first-min-lemma}
The function $u$ minimizes
\[
v\mapsto
\frac1p\int_\Omega|Dv|^p
+\int_\Omega\psi G(v)
-\int_\Omega fv
\]
over $\mathcal D_\psi$.
\end{lemma}

\begin{proof}
For every $n$, the construction of the approximate solution implies that
$u_n$ minimizes
\[
v\mapsto
\frac1p\int_\Omega |Dv|^p
+\int_\Omega \psi_n G(v)
-\int_\Omega f_n v.
\]
Hence, for every
$v\in\W\cap L^\infty(\Omega)$,
\begin{equation}\label{approx-min-ineq}
\frac1p\int_\Omega |Du_n|^p
+\int_\Omega \psi_nG(u_n)
-\int_\Omega f_nu_n
\le
\frac1p\int_\Omega |Dv|^p
+\int_\Omega \psi_nG(v)
-\int_\Omega f_nv.
\end{equation}

By weak lower semicontinuity,
\[
\frac1p\int_\Omega |Du|^p
\le
\liminf_{n\to\infty}
\frac1p\int_\Omega |Du_n|^p.
\]
Moreover, since $\psi_n\to\psi$ and $u_n\to u$ a.e., with
$\psi_n\ge0$ and $G\ge0$, Fatou's lemma gives
\[
\int_\Omega\psi G(u)
\le
\liminf_{n\to\infty}\int_\Omega\psi_nG(u_n).
\]
Also,
\[
f_nu_n\to fu
\qquad\text{strongly in }L^1(\Omega).
\]
Thus
\[
\frac1p\int_\Omega |Du|^p
+\int_\Omega\psi G(u)
-\int_\Omega fu
\le
\liminf_{n\to\infty}
\left[
\frac1p\int_\Omega |Du_n|^p
+\int_\Omega\psi_nG(u_n)
-\int_\Omega f_nu_n
\right].
\]

On the other hand, if $v\in\W\cap L^\infty(\Omega)$ is fixed, then
$G(v)\in L^\infty(\Omega)$. Since $\psi_n\to\psi$ strongly in $L^1(\Omega)$
and $f_n\to f$ strongly in $L^1(\Omega)$,
\[
\int_\Omega\psi_nG(v)\to\int_\Omega\psi G(v),
\qquad
\int_\Omega f_nv\to\int_\Omega fv.
\]
Passing to the limit in \eqref{approx-min-ineq} therefore yields
\[
\frac1p\int_\Omega |Du|^p
+\int_\Omega\psi G(u)
-\int_\Omega fu
\le
\frac1p\int_\Omega |Dv|^p
+\int_\Omega\psi G(v)
-\int_\Omega fv
\]
for every $v\in\W\cap L^\infty(\Omega)$.

Finally, let $v\in\mathcal D_\psi$ and set $v_k=T_k(v)$. Then
$v_k\to v$ strongly in $\W$, while
\[
|fv_k|\le |fv|\in L^1(\Omega),
\]
so $fv_k\to fv$ in $L^1(\Omega)$. Moreover,
$|v_k|\uparrow|v|$ and, since $G$ is even and nondecreasing on
$[0,\infty)$,
\[
G(v_k)\uparrow G(v).
\]
Since $\psi\ge0$, the monotone convergence theorem gives
\[
\int_\Omega\psi G(v_k)\to\int_\Omega\psi G(v).
\]
Passing to the limit as $k\to\infty$ proves the result.
\end{proof}

Hence
\begin{equation}\label{right-saddle}
\J(u,\psi)\le\J(v,\psi)
\qquad\forall v\in\mathcal D_\psi.
\end{equation}

Define
\[
\mathcal F_u(\varphi)
=-\frac1p\int_\Omega|D\varphi|^p
+\int_\Omega\varphi^+G(u).
\]

\begin{lemma}\label{second-max-lemma}
The function $\psi$ maximizes $\mathcal F_u$ over $\mathcal E_u$.
\end{lemma}

\begin{proof}
First observe that $\psi\in\mathcal E_u$. Indeed, since $\psi\ge0$,
testing \eqref{weak-second} with $\psi$ gives
\[
\int_\Omega \psi G(u)
=
\int_\Omega |D\psi|^p
<\infty.
\]

Let $\varphi\in\mathcal E_u$. Since
\[
|D\varphi|^p
=
|D\varphi^+|^p+|D\varphi^-|^p
\quad\text{a.e.},
\]
we have
\[
\mathcal F_u(\varphi)
=
\mathcal F_u(\varphi^+)
-\frac1p\int_\Omega|D\varphi^-|^p
\le
\mathcal F_u(\varphi^+).
\]
Moreover, $\varphi^+\in\mathcal E_u$. It is therefore enough to consider
$\varphi\ge0$.

By convexity of $\xi\mapsto |\xi|^p/p$,
\[
\frac1p|D\varphi|^p
\ge
\frac1p|D\psi|^p
+|D\psi|^{p-2}D\psi\cdot(D\varphi-D\psi).
\]
Integrating and using \eqref{weak-second} with the admissible test function
$\varphi-\psi\in\W$, we obtain
\[
\frac1p\int_\Omega|D\varphi|^p
\ge
\frac1p\int_\Omega|D\psi|^p
+\int_\Omega G(u)(\varphi-\psi).
\]
Here all terms are finite because
$\varphi G(u),\psi G(u)\in L^1(\Omega)$. Rearranging yields
\[
-\frac1p\int_\Omega|D\varphi|^p+\int_\Omega\varphi G(u)
\le
-\frac1p\int_\Omega|D\psi|^p+\int_\Omega\psi G(u),
\]
that is,
\[
\mathcal F_u(\varphi)\le\mathcal F_u(\psi).
\]
\end{proof}

Thus
\begin{equation}\label{left-saddle}
\J(u,\varphi)\le\J(u,\psi)
\qquad\forall\varphi\in\mathcal E_u.
\end{equation}
Combining \eqref{right-saddle} and \eqref{left-saddle} finishes the proof
Theorem~\ref{main-thm}.

\section{Final remarks}

The argument uses no explicit expression for $G$. The decisive properties are
convexity, positivity, evenness, and the comparison
\[
G(t)\lesssim1+|G'(t)|
\quad\text{at infinity}.
\]
The latter has two roles. First, the truncated-sign estimate for the first
equation controls the source in the second equation and yields the uniform
$W^{1,p}$ estimate for $\psi_n$. Second, together with convexity and
nontriviality, it forces at least exponential growth of $G$ at infinity. 

The estimate
\[
\int_{\{|u_n|>k\}}G(u_n)\le Ck^{-1/p'}
\]
is independent of the precise growth of $G$ and follows only from the coupled
structure and the $p$-growth of the principal operators. When $p=2$, it
reduces to the square-root tail estimate in the exponential semilinear case, see \cite{BO2026}.

%

\end{document}